\documentclass{kms-j}

\newcommand{\publname}{}
\issueinfo{}
  {}
  {}
  {}
\pagespan{1}{}
\copyrightinfo{}
  {}

\usepackage{color,graphicx}
\usepackage{amsmath,amssymb,amsthm,amsfonts}

\usepackage[all]{xy}

\newtheorem{theorem}{Theorem}

\newtheorem{lemma}{Lemma}
\newtheorem{cor}{Corollary}

\newcommand{\E}{\mathcal{E}}
\newcommand{\A}{\mathcal{A}}

\begin{document}

\title[Certain numbers on the groups of self-homotopy
equivalences] {Certain numbers on the groups of self-homotopy
equivalences}

\author{Ho Won Choi}
\address{Department of Mathematics\\Korea University\\Seoul 702-701, Korea}
\email{howon@korea.ac.kr}
\thanks{This work was supported by a Korea University Grant}

\author{Kee Young Lee}
\address{Department of Information and Mathematics\\Korea University\\Sejong 425-791, Korea}
\email{keyolee@korea.ac.kr}


\begin{abstract}
For a connected based space $X$, let $[X,X]$ be the set of all
based homotopy classes of base point preserving self map of $X$
and let $\E(X)$ be the group of self-homotopy equivalences of $X$.
We denote by $\A_{\sharp}^k(X)$ the set of homotopy classes of
self-maps of $X$ that induce an automorphism of $\pi_i(X)$ for
$i=0,1,\cdots,k$. That is, $[f]\in \A_{\sharp}^k(X)$ if and only
if $\pi_i(f):\pi_i(X)\to\pi_i(X)$ is an isomorphism for
$i=0,1,\cdots ,k$. Then, $\E(X)\subseteq\A_{\sharp}^k(X)\subseteq
[X,X]$ for a nonnegative integer $k$. Moreover, for a connected
CW-complex $X$, we have $\E(X)=\A_{\sharp}(X)$. In this paper, we
study the properties of $\A_{\sharp}^k(X)$ and discuss the
conditions under which $\E(X)=\A_{\sharp}^k(X)$ and the minimum
value of such $k$. Furthermore, we determine the value of $k$ for
various spaces, including spheres, products of spaces, and Moore
spaces.

\end{abstract}
\maketitle

\section{Introduction}\label{section1}

For a connected space $X$, we denote by $[X,X]$ the set of all
based homotopy classes of self maps of $X$. Then, $[X,X]$ is a
monoid with multiplication given by the composition of homotopy
classes. Let $\E(X)$ be the set of self-homotopy equivalences of
$X$. Then, $\E(X)$ is a group with the operation given by the
composition of homotopy classes. $\mathcal{E}(X)$ has been studied
extensively by various authors, including Arkowitz\cite{A1},
Maruyama\cite{AM}, Lee\cite{L}, Oka\cite{O}, Rutter\cite{R},
Sawashita\cite{Sa}, and Sieradski\cite{Si}. Several subgroups of
$\E(X)$ have also been studied. One of these is
$\mathcal{E}^n_{\sharp}(X)$ which consists of the self-homotopy
equivalences that induce the identity homomorphism on homotopy
groups $\pi_i(X)$, for $i=1,2,\cdots,n$. That is,
$\mathcal{E}^n_{\sharp}(X)=\{f\in \mathcal{E}(X)~|~\pi_i
(f)=id_{\pi_i(X)}~on~ \pi_i(X)~for~i=0,\cdots,n\},$ where $\pi_i
(f):\pi_i(X)\to \pi_i(X)$ is the induced homomorphism of a map
$f:X\to X$ and $n=\infty$ is available. When $n=\infty$, we simply
denote $\mathcal{E}^{\infty}_{\sharp}(X)$ as
$\mathcal{E}_{\sharp}(X)$. It is well known that
$\mathcal{E}(S^n)=Z_2$ and $\mathcal{E}_{\sharp}(S^n)=1$, where
$S^n$ is the $n$-dimensional sphere. Moreover,
$\mathcal{E}^n_{\sharp}(X)$ is a subgroup of $\mathcal{E}(X)$ and,
if $n\leq m$, $\mathcal{E}^m_{\sharp}(X) \subseteq
\mathcal{E}^n_{\sharp}(X).$ Thus, we have the following chain by
inclusion:

$$\mathcal{E}_{\sharp}(X)\subseteq\cdots\subseteq
\mathcal{E}^n_{\sharp}(X)\subseteq\cdots\subseteq
\mathcal{E}^1_{\sharp}(X)\subseteq\mathcal{E}(X).$$

We denote by $\A_{\sharp}^k(X)$ the set of homotopy classes of
self-maps of $X$ that induce an automorphism of $\pi_i(X)$ for
$i=0,1,\cdots,k$. That is, $[f]\in \A_{\sharp}^k(X)$ if and only
if $\pi_i(f):\pi_i(X)\to\pi_i(X)$ is an isomorphism for
$i=0,1,\cdots ,k$. Then, for any nonnegative integer $k$,
$\A_{\sharp}^k(X)$ is a submonoid of $[X,X]$ and always contains
$\E(X)$. $n=\infty$ is available and, if $n=\infty$, we simply
denote $\A_{\sharp}^{\infty}(X)$ as $\A_{\sharp}(X)$. If $k<n$,
then $\A_{\sharp}^ n(X)\subseteq \A_{\sharp}^k(X)$; thus, we have
the following chain by inclusion:
\begin{center}
$\E(X)\subseteq\A_{\sharp}(X)\subseteq... \subseteq
\A_{\sharp}^1(X)\subseteq \A_{\sharp}^0(X)=[X,X].$
\end{center}

For any connected CW-complex $X$, we have $\A_{\sharp}(X)=\E(X)$.
These two facts give rise to the question of what is the largest
$\A_{\sharp}^ n(X)$ that is equal to $\E(X)$. In other words, what
is the minimum $k$ such that $\E(X)=\A_{\sharp}^k(X)$? We denote
by $N\E(X)$ the least nonnegative integer $k$ such that
$\E(X)=\A_{\sharp}^k(X)$, and call it {\it the self-closeness
number of } $X$. That is,
$$N\E(X) =
min\{k~|~\mathcal{A}^k_{\sharp}(X)=\mathcal{E}(X) ~for~ k\geq
0\}.$$

In this paper, we investigate properties of the self-closeness
number of a space including the homotopy invariance, the relation
with the dimension of spaces and the relation with the
$n$-connectivity. From the results of this investigation, we
completely determine the value of $k$ in several specific cases.
Moreover, we give an equivalence relation on $\A_{\sharp}^k(X)$
and study the relations between the set of equivalence classes and
$\E(X)$, $\E_{\sharp}^n(X)$, or the self-closeness numbers.

Throughout this paper, all topological spaces are based and have
the based homotopy type of CW-complex. All maps and homotopies
will preserve the base points. For the spaces $X$ and $Y$, we
denote by $[X,Y]$ the set of homotopy classes of maps from $X$ to
$Y$. We do not distinguish between the notation of a map $X \to Y$
and that of its homotopy class in $[X, Y]$.

\section{Properties of $\A_{\sharp}^k(X)$ and the self-closeness number }\label{section2}

We begin by identifying $\A_{\sharp}^k(X)$ as a submonoid of
$[X,X]$.

\begin{lemma}
$\A_{\sharp}^k(X)$ is a submonoid of $[X,X]$ for any nonnegative
integer $k$.
\begin{proof}
Let $f, g\in \A_{\sharp}^k(X)$. Then, $\pi_i(f),\pi_i(g)\in
Aut(\pi_i(X))$ for $i=0,1,\cdots, k$, where $Aut(\pi_i(X))$ is the
automorphism group of $\pi_i(X)$. Thus, $\pi_i(f\circ
g)=\pi_i(f)\circ\pi_i(g)\in Aut(\pi_i(X))$ for $i=0,1,\cdots, k$.
It follows that $f\circ g\in \A_{\sharp}^k(X)$. Clearly, $f\circ
(g\circ h)=(f\circ g)\circ h$. Moreover, because
$\pi_i(id_X)=id_{\pi_i(X)}\in Aut(\pi_i(X))$ for $i=0,1,\cdots,
k$, $id_X\in \A_{\sharp}^k(X)$.
\end{proof}

\end{lemma}

In general, $\E(X)\neq \A_{\sharp}(X)$.

\

\noindent {\bf Example 1.} Let $X$ be a quasi-circle in the plane
$\mathbb{R}^2$; that is, $$X=\{(x,y)|~y=sin(\frac{1}{x}), 0< x\leq
\frac{1}{\pi}\}\cup\{(0,y)|-2\leq y\leq 1\}$$ $$\cup\{(x,-2)|0\leq
x\leq \frac{1}{\pi}\}\cup\{(\frac{1}{\pi},y)|-2\leq y\leq 0\}.$$

Then, for each $x_0\in X$, the inclusion map $i:x_0 \to X$ is a
weak homotopy equivalence, but not a homotopy equivalence. Thus,
the constant map $C:X\to X$ given by $C(x)=x_0$ is not a homotopy
equivalence. Therefore, $C$ does not belong to $\E(X)$. However,
as $\pi_i(X)=0$ for all $i$, $\pi_i(C)\in Aut(\pi_i(X))$ for all
$i$.

\

Whitehead theorem ensures that $\E(X)= \A_{\sharp}(X)$ provided
that $X$ is a CW-complex.

\begin{lemma}
If $X$ is a CW-complex, then $\E(X)= \A_{\sharp}(X)$.
\end{lemma}

In general, $\A_{\sharp}^{k-1}(X)\neq\A_{\sharp}^k(X)$.

\

\noindent{\bf Example 2.} Let $M(\mathbb{Z}_4,5)$ be a Moore
space. Then, by the results in \cite{AT},
$[M(\mathbb{Z}_4,5),M(\mathbb{Z}_4,5)]\cong \mathbb{Z}_4\oplus
\mathbb{Z}_2$. Since $M(\mathbb{Z}_4,5)$ is 4-connected,
$\pi_i(M(\mathbb{Z}_4,5))=0$ for $ i=0,1,2,3, 4$. Thus,
$\A_{\sharp}^{4}(M(\mathbb{Z}_4,5))=[M(\mathbb{Z}_4,5),M(\mathbb{Z}_4,5)]\cong
\mathbb{Z}_4\oplus \mathbb{Z}_2$ and, for
$f\in[M(\mathbb{Z}_4,5),M(\mathbb{Z}_4,5)]$,
 $f\in\A_{\sharp}^5(M(\mathbb{Z}_4,5))$ if and only if $\pi_5(f)\in
 Aut(\pi_5(M(\mathbb{Z}_4,5)))$.
Therefore,
$\A_{\sharp}^5(M(\mathbb{Z}_4,5))\cong\mathbb{Z}_4^{\ast}\times\mathbb{Z}_2$,
where $\mathbb{Z}_4^{\ast}$ is the automorphism group of
$\mathbb{Z}_4$. Rutter \cite{R} showed that
$\E(M(\mathbb{Z}_4,5))\cong
\mathbb{Z}_4^{\ast}\times\mathbb{Z}_2$. Consequently,
$$\A_{\sharp}^5(M(\mathbb{Z}_4,5))=\E(M(\mathbb{Z}_4,5))\cong\mathbb{Z}_4^{\ast}\times\mathbb{Z}_2
\cong\mathbb{Z}_2\times\mathbb{Z}_2.$$ Therefore, we have
$$\E(M(\mathbb{Z}_4,5))=\A_{\sharp}^5(M(\mathbb{Z}_4,5))\varsubsetneq
\A_{\sharp}^{4}(M(\mathbb{Z}_4,5))=[M(\mathbb{Z}_4,5),M(\mathbb{Z}_4,5)].$$

\

Example 2 ensures that the self-closeness number
$N\E(M(\mathbb{Z}_4,5)$ is 5. We will discuss details of the more
general Moore spaces $M(G,n)$ where $G$ is an abelian group in
Section 3.

The self-closeness number is a homotopy invariant.
\begin{theorem}
Let $X$ and $Y$ be CW-complexes. If $X$ and $Y$ have the same
homotopy type, then $N\E(X)=N\E(Y)$.

\end{theorem}

 Before we prove
Theorem 1, we first show the following lemma .

\begin{lemma}
If $X$ and $Y$ have the same homotopy type, then there is a
one-to-one correspondence
$\psi:\A_{\sharp}^n(X)\to\A_{\sharp}^n(Y)$ that is a monoid
homomorphism for each nonnegative integer $n$. Moreover, the
restriction of $\psi$ to $\E(X)$ is a group homomorphism into
$\E(Y)$.

\begin{proof}
Let $\alpha:X\to Y$ be a homotopy equivalence with homotopy
inverse $\beta$. Define $\psi:\A_{\sharp}^n(X)\to\A_{\sharp}^n(Y)$
by $\psi (f)=\alpha \circ f\circ \beta$ for each $f\in
\A_{\sharp}^n(X)$.

Then, $\psi$ is well-defined, because $f_1\simeq f_2\Rightarrow
\alpha \circ f_1\circ \beta\simeq\alpha \circ f_2\circ \beta$ and
$\pi_i(\alpha \circ f\circ \beta)=\pi_i(\alpha) \circ \pi_i
(f)\circ \pi_i(\beta)$ is an isomorphism on $\pi_i(Y)$ for
$i=0,\cdots, n$.

If $\psi (f_1)=\psi (f_2)$, then $\alpha \circ f_1\circ
\beta\simeq\alpha \circ f_2\circ \beta$. This implies $f_1\simeq
f_2$. Thus, $\psi$ is one-to-one.

For $h\in \A_{\sharp}^n(Y)$, $\beta\circ h \circ\alpha\in
\A_{\sharp}^n(X)$ and $\psi (\beta\circ h
\circ\alpha)=\alpha\circ(\beta\circ h \circ\alpha)\circ \beta\simeq h$. Thus,
$\psi$ is onto.

Since
$$\psi(f\circ g)= \alpha
\circ (f\circ g)\circ \beta \simeq \alpha \circ (f\circ
\beta\circ\alpha\circ g)\circ \beta\simeq (\alpha \circ f\circ
\beta)\circ(\alpha \circ g\circ \beta)=\psi(f)\circ\psi(g),$$ $\psi$ is
a monoid homomorphism.

If $f\in\E(X)$, then $\alpha \circ f\circ \beta\in\E(Y)$. Thus,
$\psi:\E(X)\to\E(Y)$ is an isomorphism.

\end{proof}

\end{lemma}

\begin{proof} [Proof of Theorem 1]
It is sufficient to show that $\E(X)=\A_{\sharp}^n(X)$ if and only
if $\E(Y)=\A_{\sharp}^n(Y)$ for each nonnegative integer $n$. As
$\E(X)\subseteq\A_{\sharp}^n(X)$ and
$\E(Y)\subseteq\A_{\sharp}^n(Y)$, we will show
$\A_{\sharp}^n(X)\subseteq\E(X)$ if and only if
$\A_{\sharp}^n(Y)\subseteq\E(Y)$. Suppose
$\A_{\sharp}^n(X)\subseteq\E(X)$ and
$\psi:\A_{\sharp}^n(X)\to\A_{\sharp}^n(Y)$ is the one-to-one
correspondence in Lemma 3. If $f\in \A_{\sharp}^n(Y)$, then
$\psi^{-1}(f)\in \A_{\sharp}^n(X)\subseteq\E(X)$. By Lemma 3,
$\psi(\psi^{-1}(f))\in \E(Y)$. Thus, $f\in\E(Y)$. Consequently,
$\A_{\sharp}^n(Y)\subseteq\E(Y)$.

Similarly, if $\A_{\sharp}^n(Y)\subseteq\E(Y)$, then
$\A_{\sharp}^n(X)\subseteq\E(X)$.

\end{proof}

In Lemma 2, we showed that $M(\mathbb{Z}_4,5)$ is 4-connected
whereas its self-closeness number is 5. This is true in general.

\begin{lemma}
If $X$ is $n$-connected and $\E(X)\neq[X,X]$, then $N\E(X)\geq
n+1$.

\begin{proof}
Assume $N\E(X)=k\leq n$. Then, $\E(X)=\A_{\sharp}^k(X)$ for $k\leq
n$. However, because $\pi_i(X)=0$ for $i=0,\cdots, n$,
$$\A_{\sharp}^n(X)=\cdots=\A_{\sharp}^{k}(X)
=\cdots=\A_{\sharp}^0(X)=[X,X].$$ This contradicts the hypothesis.
Thus, $N\E(X)\geq n+1$.

\end{proof}
\end{lemma}

\begin{cor}
Let $X$ be a CW-complex with cells of dimension $\geq n$, except
0-cells. Then, $N\E(X)\geq n$ provided that $\E(X)\neq [X,X]$.
\end{cor}

The self-closeness number of a CW-complex is closely related to
the dimension of the space.

\begin{theorem} If $X$ is a CW-complex with dimension $n$, then
$N\E(X)\leq n$.

\begin{proof}
It is sufficient to show that $\A_{\sharp}^{n}(X)=\E(X)$. Let
$f\in\A_{\sharp}^{n}(X)$. Then $\pi_i(f)\in Aut(\pi_i(X))$ for
$i=0,1,\cdots, n$. By Theorem 5.1.32 in \cite{AGP}, the induced
map $f_{\sharp}:[X,X]\to[X,X]$ is onto. Thus, there is a map
$g:X\to X$ such that $f_{\sharp}(g)=id_X$, that is, $f\circ
g\simeq id_X$. Because $\pi_i(f)\circ \pi_i(g)=\pi_i(f\circ
g)=\pi_i(id_X)=id_{\pi_i(X)}$ and $\pi_i(f)$ is an isomorphism,
$\pi_i(g)\in Aut(\pi_i(X))$ for $i=0,1,\cdots, n$. Thus,
$g_{\sharp}:[X,X]\to[X,X]$ is also onto. Therefore, there is a map
$f':X\to X$ such that $g_{\sharp}(f')=id_X$, that is, $g\circ
f'\simeq id_X$. However, as $f\circ (g\circ f')\simeq f\circ
id_X\simeq f$ and $f\circ g\simeq id_X$, $f'\simeq f$. Thus
$g\circ f \simeq g\circ f'\simeq id_X$. Therefore, $g$ is a
homotopy inverse of $f$, which implies that $f\in \E(X)$.
\end{proof}

\end{theorem}

\begin{cor} Let $S^k$ be the $k$-dimensional sphere. Then,
$N\E(S^k)=k$.

\begin{proof}
As $dim(S^k)=k$, $N\E(S^k)\leq k$ by Theorem 2. It is well known
that $$\E(S^k)\cong \mathbb{Z}_2 \neq \mathbb{Z}\cong [S^k,
S^k].$$ However, because $S^k$ is $(k-1)$-connected, $N\E(S^k)\geq
k$ by Lemma 4. Thus, $N\E(S^k)=k$.
\end{proof}

\end{cor}

In the above corollary, the self-closeness number of a sphere is
just its dimension. However, this is not true in general. Before
we give an example, we discuss the self-closeness numbers of
product spaces.

\begin{theorem} Let $X$ and $Y$ be CW-complexes. Then, we have
$$ N\E(X\times Y)\geq max\{N\E(X), N\E(Y)\}.$$

\begin{proof} Assume that $\E(X\times Y)=\A_{\sharp}^{k}(X\times
Y)$ for the nonnegative integer $k<max\{N\E(X), N\E(Y)\}$. We may
assume that $max\{N\E(X), N\E(Y)\}=N\E(X)$. Let
$f\in\A_{\sharp}^{k}(X)$ and $g\in\A_{\sharp}^{k}(Y)$. Then,
$\pi_i(f)\in Aut(\pi_i(X))$ and $\pi_i(g)\in Aut(\pi_i(Y))$ for
$i=0, \cdots, k$. Thus, $\pi_i(f)\times\pi_i(g)\in
Aut(\pi_i(X))\times Aut(\pi_i(Y))$. Let $F:\pi_i(X\times Y)\to
\pi_i(X)\times \pi_i(Y) $ be an isomorphism given by
$F[r]=([p_1r],[p_2r])$ for $r:S^i\to X\times Y$, where $p_i$ is
the projection to the $i$-th factor for $i=1,2$. Then, we have the
following commutative diagram:
$$\xymatrix@C=19mm{
    \pi_i(X\times Y) \ar[d]^{F} \ar[r]^{\pi_i(f\times g)}
    & \pi_i(X\times Y)\ar[d]^{F}\\
    \pi_i(X)\times \pi_i(Y)   \ar[r]^{\pi_i(f)\times\pi_i(g) }
    &\pi_i(X)\times \pi_i(Y).
\\
} $$ It follows that $\pi_i(f\times g)$ is an isomorphism for
$i=0,\cdots, k$. Thus,
 $f\times g \in \A_{\sharp}^{k}(X\times
Y)=\E(X\times Y)$. Therefore, $\pi_i(f\times g)$ is an isomorphism
for each nonnegative integer $i$, and so $\pi_i(f)\times\pi_i(g)$
is also an isomorphism for each nonnegative integer $i$.  Thus,
$\pi_i(f)$ and $\pi_i(g)$ are isomorphisms for each nonnegative
integer $i$. Consequently, $f\in\E(X)$ and $g\in\E(Y)$. As a
result, $\E(X)=\A_{\sharp}^{k}(X)$ and $\E(Y)=\A_{\sharp}^{k}(Y)$.
This contradicts the minimality of $N\E(X)$.

\end{proof}

\end{theorem}

\noindent{\bf Example 3.} By the results in \cite{Si}, $[S^1\times
S^3, S^1\times S^3]\cong \mathbb{Z}\oplus \mathbb{Z}\oplus
\mathbb{Z}_2$ and $\E(S^1\times S^3)\cong \mathbb{Z}_2\oplus
\mathbb{Z}_2\oplus \mathbb{Z}_2$. Thus, by the above theorem,
$$N\E(S^1\times S^3)\geq max\{N\E(S^1),N\E(S^3)\}=3.$$
However, as $dim(S^1\times S^3)=4$, $N\E(S^1\times S^3)\leq 4$.
Therefore, we have $N\E(S^1\times S^3)=3\text{ or }4$. If we give
a more detailed computation, we conclude that $N\E(S^1\times
S^3)=3$. As the dimension of the space is 4, the self-closeness
number is not equal to the dimension.

\section{ The self-closeness number of a Moore space }\label{section 3}

In this section, we determine the self-closeness numbers of
general Moore spaces.

Given an abelian group $G$ and an integer $n\geq 3$, let $M(G,n)$
be a Moore space. Then we obtain the following theorem.

 \begin{theorem}
 $\A_\sharp^n(M(G,n))=\E(M(G,n))$.
 \begin{proof}
Since $\E(M(G,n)) \subseteq \A_\sharp^n(M(G,n))$ in general,
 it is sufficient to show that $\A_\sharp^n(M(G,n)) \subseteq \E(M(G,n))$.
 Since $M(G,n)$ is $(n-1)$-connected space, the Hurewicz homomorphism
 $h_n : \pi_n(M(G,n))\to H_n(M(G,n))$ is an isomorphism.
 For any $f\in [M(G,n),M(G,n)]$, the diagram
 $$\xymatrix@C=10mm{
 \pi_n(M(G,n)) \ar[r]^{\pi_n(f)} \ar[d]_{h_n} &  \pi_n(M(G,n)) \ar[d]^{h_n}\\
 H_n(M(G,n))  \ar[r]_{H_n(f)} & H_n(M(G,n))}$$ commutes.
 Let $f$ be an element of $\A_\sharp^n(M(G,n))$. Then $\pi_i(f)$ is an isomorphism for $i\leq n$.
 Since $h_n$ is an isomorphism, $H_n(f)$ is an isomorphism.
 Since the Moore space has nontrivial homology only for dimension $n$, $f$ is a homotopy equivalence
 by the homology version of the Whitehead Theorem. Hence $\A_\sharp^n(M(G,n)) \subseteq  \E(M(G,n))$.
 \end{proof}
 \end{theorem}

Since $M(G,n)$ is $(n-1)$-connected space, $\A_\sharp^i(M(G,n)) =
[M(G,n),M(G,n)]$ for
 $i<n$. Moreover, since
 $[M(G,n),M(G,n)]\neq \E(M(G,n))$, $n$ is the minimum number in
 the set $\{k~|~\E(M(G,n)) = \A_\sharp^k(M(G,n))\}$. Thus, we have
 the following corollary.

 \begin{cor} $N\E(M(G,n))=n$ for $n\geq 3$.
 \end{cor}

In \cite{R}, Rutter showed that
$\mathcal{E}(M(\mathbb{Z}_q,n))=\mathbb{Z}_{(2,q)}\times
\mathbb{Z}_q^\ast$ where  $\mathbb{Z}_q^\ast$ is the automorphism group of
$\mathbb{Z}_q$ and $(2,q)$ is the greatest common divisor
of $2$ and $q$. From Theorem 4,
$\mathcal{A}^n_\sharp(M(\mathbb{Z}_q,n))=\mathcal{E}(M(\mathbb{Z}_q,n))$
for all $q>1$. Moreover, since
$\mathcal{E}(M(\mathbb{Z}_q,n))\subset\mathcal{A}^k_\sharp(M(\mathbb{Z}_q,n))
\subset\mathcal{A}^n_\sharp(M(\mathbb{Z}_q,n))$ if $n \leq k$, we have\\

\begin{center} $\mathcal{A}^k_\sharp(M(\mathbb{Z}_q,n)) = $
$\left\{\begin{matrix}
~~[M(\mathbb{Z}_q,n),M(\mathbb{Z}_q,n)]~&~\text{if}~k< n,\\
~~\mathbb{Z}^\ast_q~~& ~\text{if}~~k\geq n,~~ q~~ \text{is odd},\\
~~\mathbb{Z}_2\times\mathbb{Z}^\ast_q~~& ~\text{if}~~k\geq n,~~ q~~ \text{is even}.\\
\end{matrix}\right.$\\
\end{center}

\section{ An equivalence relation on $\mathcal{A}^n_\sharp(X)$ }\label{section 4}

We define a relation '$\simeq_n$' on $\mathcal{A}^n_\sharp(X)$ as
follows: for $f,g\in \mathcal{A}^n_\sharp(X)$, $f\simeq_n g$ if
$\pi_i(f)=\pi_i(g):\pi_i(X)\rightarrow \pi_i(X)$ for $i=0,\cdots,
n$. If $f\simeq_n g$, then we say $f$ {\it is} $n${\it -self
homotopic to} $g$. By definition, '$\simeq_n$' is an equivalence
relation. We denote by $\bar{f}$ the equivalence class of $f$ in
$\mathcal{A}^n_\sharp(X)$, and call this the $n$-{\it self
homotopy equivalence class of $f$ on $X$}. Moreover, we denote by
$\bar{\A}_\sharp^n(X)$ the set of all $n$-self homotopy
equivalence classes on $X$. $n=\infty$ is available, and we denote
$\bar{\A}_\sharp^\infty(X)$ as simply $\bar{\A}_\sharp(X)$.

In general, $\bar{\A}_\sharp^n(X)\neq \A_\sharp^n(X)$.

\

\noindent {\bf Example 4.} By \cite{AOS}, It was shown that
$\E_\sharp(S^1\times S^3)=\mathbb{Z}_2$. However,
$\A_\sharp(S^1\times S^3)=\E(S^1\times
S^3)=\mathbb{Z}_2\oplus\mathbb{Z}_2\oplus\mathbb{Z}_2$. Choose two
distinct elements $\alpha, \beta\in \E_\sharp(S^1\times
S^3)=\mathbb{Z}_2$. Then since
$\pi_i(\alpha)=id_{\pi_i(X)}=\pi_i(\beta)$,
$\bar{\alpha}=\bar{\beta}$ in $\bar{\A}_\sharp(S^1\times S^3)$.

\begin{theorem}
Let $X$ be a CW-complex.  Then, $\bar{\A}_\sharp^n(X)$ is a monoid
for $n<\infty$, and $\bar{\A}_\sharp(X)$ is a group.

\begin{proof} Let us define $\bar{f}\cdot\bar{g}=\overline{f\circ g}$ for
$f,g\in \bar{\A}_\sharp^n(X)$. Then, the operation '$\cdot$' is
well-defined and associative. Furthermore, $\overline{id_X}$ is an
element of $\bar{\A}_\sharp^n(X)$. Thus, $\bar{\A}_\sharp^n(X)$
has a monoid structure.

As $\A_\sharp(X)=\E(X)$ by Lemma 2, for each
$\bar{f}\in\A_\sharp(X)$, its representative $f$ belongs to
$\E(X)$. Thus, $f$ has the homotopy inverse $g$ in
$\E(X)=\A_\sharp(X)$. Then, $\bar{g}$ is the inverse element of
$\bar{f}$. As a result, $\bar{\A}_\sharp(X)$ has a group
structure.

\end{proof}
\end{theorem}

\begin{cor} Let $X$ be a CW-complex. If $N\E(X)=n$, then $\bar{\A}_\sharp^k(X)$ has a group
structure for $k\geq n$.
\end{cor}

Consider the $n$-self equivalence class $\bar{f}$ of $f\in
\A_\sharp^n(X)$. Then, $\bar{f}$ is a subset of $\A_\sharp^n(X)$.
In particular, $\overline{id_X}$ is a submonoid of
$\A_\sharp^n(X)$. For a given $g\in \A_\sharp^n(X)$, we define a
set $g\cdot \bar{f}$ as follows:
$$g\cdot\bar{f}:=\{g\circ\alpha|\alpha\in\bar{f}\}.$$
Similarly, we define the set $\bar{f}\cdot g$ as $$\bar{f}\cdot
g:=\{\beta\circ g|\beta\in\bar{f}\}.$$ In general, $g\cdot
\bar{f}\subseteq \overline{g\circ f}$, $\bar{f}\cdot g \subseteq
\overline{f\circ g }$, and $g\cdot \bar{f}\neq \bar{f}\cdot g$.
Moreover, $\E_\sharp^n(X)\subseteq \overline{id_X}\subseteq
\A_\sharp^n(X)$ for each positive integer $n$. In fact, if $f\in
\E_\sharp^n(X)$, then $\pi_i(f)=id_{\pi_i(X)}$ for $i=0,\cdots,n$.
However, $\pi_i(f)=id_{\pi_i(X)}=\pi_i(id_X)$. Thus, $f\in
\overline{id_X}$.

\begin{lemma} $\E_\sharp^n(X)$ is a normal subgroup of $\E(X)$ for
all $n$.

\begin{proof} Let $h\in\E_\sharp^n(X)$ and $f\in\E(X)$.
Then, for the homotopy inverse $g$ of $f$, we have
$$\pi_i(f\circ h\circ g)=\pi_i(f)\circ\pi_i( h)\circ \pi_i(g)=\pi_i(f)\circ
\pi_i(g)=\pi_i(id_X)=id_{\pi_i(X)}$$ for $i=0,\cdots,n$. Thus,
$f\circ h\circ g\in\E_\sharp^n(X)$.
\end{proof}

\end{lemma}

By the above lemma, we can consider the factor group
$\E(X)/\E_\sharp^n(X)$. Here we investigate the relation between
this factor group and the $n$-equivalence classes.

\begin{theorem} $\E(X)/\E_\sharp^n(X)$ is isomorphic to a subgroup
of $\bar{\A}_\sharp^n(X)$ for each positive integer $n$. In
particular, if $N\E(X)=n$, then $\bar{\A}_\sharp^n(X)\cong
\E(X)/\E_\sharp^n(X)$.

\begin{proof}
Define $\Psi:\E(X)/\E_\sharp^n(X)\rightarrow \bar{\A}_\sharp^n(X)$
by $\Psi(f\cdot \E_\sharp^n(X))=\bar{f}$. If
$f\cdot\E_\sharp^n(X)=g\cdot \E_\sharp^n(X)$ for $f,g\in \E(X)$,
then $g^{-1}\circ f\in\E_\sharp^n(X)$. Thus,
$id_{\pi_i(X)}=\pi_i(g^{-1}\circ f)=\pi_i(g^{-1})\circ \pi_i(f)$
for $i=0,\cdots,n$. It follows that $\pi_i(g)= \pi_i(f)$ for
$i=0,\cdots,n$ and $\bar{f}=\bar{g}$. Thus, we conclude that
$\Psi$ is well-defined.

If $\Psi(f\cdot \E_\sharp^n(X))=\Psi(g\cdot \E_\sharp^n(X))$, then
$\bar{f}=\bar{g}$ in $\bar{\A}_\sharp^n(X)$. Thus, $\pi_i(g)=
\pi_i(f)$ for $i=0,\cdots,n$. Therefore,
$id_{\pi_i(X)}=\pi_i(g)^{-1}\circ \pi_i(f)=\pi_i(g^{-1}\circ f)$
for $i=0,\cdots,n$. It follows that $g^{-1}\circ
f\in\E_\sharp^n(X)$. That is, $f\cdot\E_\sharp^n(X)=g\cdot
\E_\sharp^n(X)$. Thus, $\Psi$ is one-to-one.

Furthermore, since $\Psi((f\cdot \E_\sharp^n(X))\cdot(g\cdot
\E_\sharp^n(X)))=(f\circ g)\cdot \E_\sharp^n(X) =\overline{f\circ
g}=\bar{f}\cdot\bar{g}=\Psi((f\cdot
\E_\sharp^n(X))\cdot\Psi((g\cdot \E_\sharp^n(X))$, $\Psi$ is a
monomorphism.

In the case that $N\E(X)=n$, $\Psi$ is onto. In fact, if
$\bar{f}\in\bar{\A}_\sharp^n(X)$, then $f\in \A_\sharp^n(X)=\E(X)$
and $\Psi(f\cdot \E_\sharp^n(X))=\bar{f}$.

\end{proof}

\end{theorem}

\begin{cor} $\bar{\A}_\sharp(X)\cong \E(X)/\E_\sharp(X)$.
\end{cor}

\begin{cor} If $\E_\sharp^n(X)=\{id_X\}$ for some positive integer
$n$, $\bar{\A}_\sharp(X)\cong \E(X)$.
\begin{proof} Let $\E_\sharp^n(X)=\{id_X\}$. Since
$\E_\sharp(X)\subseteq\cdots\subseteq \E_\sharp^n(X)$,
$\E_\sharp(X)=\{id_X\}$. Thus, $\bar{\A}_\sharp(X)\cong \E(X)$ by
Corollary 5.
\end{proof}

\end{cor}

\

\noindent{\bf Example 5.} In \cite{CL}, It was shown that
$\E_\sharp^{n+2}(M(\mathbb{Z}_q, n+1)\vee (M(\mathbb{Z}_p,
n))=\{id_X\}$ for $n>4$ if $q$ and $p$ are odd. Thus,
$$\bar{\A}_\sharp(M(\mathbb{Z}_q, n+1)\vee (M(\mathbb{Z}_p,
n))\cong \E(M(\mathbb{Z}_q, n+1)\vee (M(\mathbb{Z}_p, n))$$ by
Corollary 6.\\

\noindent{\bf Acknowledgements}\\

We would like to thank the referee(s) for offering other suggestions and comments
that improved the quality of the paper.

\end{document}